\documentclass[11pt,twoside,reqno]{amsart}

\usepackage{version}         

\usepackage[right]{showlabels}
\usepackage{color,enumerate}

\usepackage[OT1]{fontenc}
\usepackage{type1cm}
\usepackage{amssymb}
\usepackage{mathrsfs}

\usepackage[left=2.3cm,top=3cm,right=2.3cm]{geometry}
\numberwithin{equation}{section}

\theoremstyle{plain}
\newtheorem{theorem}{Theorem}[section]

\newtheorem{proposition}[theorem]{Proposition}
\newtheorem{lemma}[theorem]{Lemma}
\newtheorem{conjecture}[theorem]{Conjecture}

\theoremstyle{remark}

\theoremstyle{definition}

\newtheorem{question}[theorem]{Question}

\newcommand{\iii}{\mathtt{i}}
\newcommand{\jjj}{\mathtt{j}}
\newcommand{\kkk}{\mathtt{k}}

\DeclareMathOperator{\dimaff}{dim_{aff}}

\usepackage[usenames,dvipsnames]{xcolor}

\begin{document}

\title[Observations on K\"aenm\"aki measures]{Some observations on K\"aenm\"aki measures}

\author{Ian D. Morris}
\address{Department of Mathematics\\
	University of Surrey\\
	Guildford GU2 7XH\\
	United Kingdom}
\email{i.morris@surrey.ac.uk}

\date{\today}

\begin{abstract}
In this note we investigate some properties of equilibrium states of affine iterated function systems, sometimes known as \emph{K\"aenm\"aki measures}. We give a simple sufficient condition for K\"aenm\"aki measures to have a gap between certain specific pairs of Lyapunov exponents, partially answering a question of B. B\'ar\'any, A. K\"aenm\"aki and H. Koivusalo. We also give sharp bounds for the number of ergodic K\"aenm\"aki measures in dimensions up to 4, answering a question of J. Bochi and the author within this range of dimensions. Finally, we pose an open problem on the Hausdorff dimension of self-affine measures which may be reduced to a statement concerning semigroups of matrices in which a particular weighted product of absolute eigenvalues is constant.
\end{abstract}

\maketitle

\section{Introduction and statement of results}

If $T_1,\ldots,T_d \colon \mathbb{R}^d \to \mathbb{R}^d$ are contractions then it is well known that there exists a unique nonempty compact set $X \subset \mathbb{R}^d$ such that $X=\bigcup_{i=1}^N T_iX$. In such a situation we call $(T_1,\ldots,T_N)$ an \emph{iterated function system} and the set $X$ its \emph{attractor}. When the transformations $T_i$ are all similitudes the set $X$ is called \emph{self-similar}, and in this case the dimension properties of the attractor have been well-understood since the 1981 work of J. E. Hutchinson \cite{Hu81}, at least in the case where the different images $T_iX$ do not too strongly overlap. In the case where the maps $T_i$ are merely affine, the dimension properties of the attractor are a topic of ongoing investigation (see for example \cite{Ba07,BaKaKo17,BaRa17,DaSi16,FaKe16,Fr12,MoSh16}). In this case an upper bound for the Hausdorff dimension was given by Falconer in 1988 (see \cite{Fa88}) and under mild additional conditions this was shown to give the exact value of the Hausdorff dimension in almost all cases in a precise sense; the focus of current research is to demonstrate that Falconer's formula for the Hausdorff dimension is valid for large \emph{explicit} families of affine iterated function systems.

In \cite{Ka04b}, A. K\"aenm\"aki introduced a class of measures on symbolic spaces which are expected to induce measures on the attractor with Hausdorff dimension equal to Falconer's bound. K\"aenm\"aki's measures have been investigated in \cite{BoMo17,FeKa11,KaMo16,KaVi10,Mo17a,Mo17b}, motivated by the ultimate goal of showing that they induce high-dimensional measures on the attractors of affine iterated function systems. In this note we shall present two results on K\"aenm\"aki measures, one addressing their Lyapunov exponents (in response to a question of B. B\'ar\'any, A. K\"aenm\"aki and H. Koivusalo) and one addressing the maximum number of distinct ergodic K\"aenm\"aki measures which a given iterated function system may have (in response to questions of A. K\"aenm\"aki, M. Vilppolainen, J. Bochi and the author). 

Let us now give the definition of a K\"aenm\"aki measure. Let $M_d(\mathbb{R})$ denote the vector space of all $d \times d$ real matrices, and for $A \in M_d(\mathbb{R})$ let $\alpha_1(A),\ldots,\alpha_d(A)$ denote the \emph{singular values} of $A$, which are defined to be the non-negative square roots of the positive semidefinite matrix $A^TA$ listed in decreasing order.
For each integer $d \geq 1$ and real number $s \geq 0$ we define the \emph{singular value function} $\varphi^s \colon M_d(\mathbb{R}) \to \mathbb{R}$, where $M_d(\mathbb{R})$ denotes the vector space of all $d \times d$ real matrices, by 
\[\varphi^s(A):=\left\{\begin{array}{cl}\alpha_1(A)\cdots \alpha_{\lfloor s \rfloor}(A) \alpha_{\lceil s \rceil}(A)^{s-\lfloor s \rfloor}&\text{if }0 \leq s \leq d,\\
|\det A|^{\frac{s}{d}}&\text{if }s \geq d.\end{array}\right.\]
For each $s \geq 0$ the singular value function satisfies $\varphi^s(AB)\leq \varphi^s(A)\varphi^s(B)$ for all $A,B \in M_d(\mathbb{R})$. If $T_1,\ldots,T_N \colon \mathbb{R}^d \to \mathbb{R}^d$ are affine contractions each having the form $T_ix=A_ix+v_i$ where $A_1,\ldots,A_N \in M_d(\mathbb{R})$ and $v_1,\ldots,v_d \in \mathbb{R}^d$, we observe that for each $s \geq 0$ the limit
\[P(\mathsf{A},\varphi^s):=\lim_{n \to \infty} \frac{1}{n}\log \sum_{i_1,\ldots,i_n=1}^N \varphi^s\left(A_{i_n}\cdots A_{i_1}\right)\]
exists by subadditivity. In this case we define the \emph{affinity dimension} of $\mathsf{A}:=(A_1,\ldots,A_N)$ to be the quantity
\[\dimaff (A_1,\ldots,A_N):=\inf\left\{s \geq 0 \colon P(\mathsf{A},\varphi^s)>0\right\}.\]
The singular value function and affinity dimension were introduced by K. Falconer in \cite{Fa88}, and their properties subsequently investigated in \cite{BoMo17,FaMi07,FaSl09,FeSh14,FrXX,KaVi10,Mo16}.

For each $N \geq 1$ let $\Sigma_N:=\{1,\ldots,N\}^{\mathbb{N}}$ which we equip with the infinite product topology. With respect to this topology $\Sigma_N$ is compact and metrisable. We let $\sigma \colon \Sigma_N \to \Sigma_N$ denote the shift transformation $\sigma\left[(x_n)_{n=1}^\infty\right]:=(x_{n+1})_{n=1}^\infty$, which is continuous. We let $\mathcal{M}_\sigma$ 
denote the set of all $\sigma$-invariant Borel probability measures on $\Sigma_N$. A measure $\nu \in \mathcal{M}_\sigma$ will be called a \emph{$\varphi^s$-equilibrium state of $\mathsf{A}=(A_1,\ldots,A_N)$} if it maximises the expression
\[h(\mu) +\lim_{n \to \infty} \frac{1}{n}\int \log \varphi^s(A_{i_n}\cdots A_{i_1})d\mu[(i_k)_{k=1}^\infty]\]
over all $\mu \in \mathcal{M}_\sigma$, where $h(\mu)$ denotes Kolmogorov-Sinai entropy. In the case where $s$ is equal to the affinity dimension of $\mathsf{A}$ a $\varphi^s$-equilibrium state of $\mathsf{A}$ is called a \emph{K\"aenm\"aki measure}.   

The first question which we investigate in this article is concerned with the number of ergodic $\varphi^s$-equilibrium states of an invertible matrix tuple $\mathsf{A}=(A_1,\ldots,A_N) \in GL_d(\mathbb{R})^N$.
Let us say that $\mathsf{A}=(A_1,\ldots,A_N) \in GL_d(\mathbb{R})^N$ is \emph{simultaneously triangularisable}, or simply \emph{triangularisable}, if there exists a basis for $\mathbb{R}^d$ with respect to which all of the matrices $A_i$ are upper triangular.
A. K\"aenm\"aki asked in \cite{Ka04b} whether for every $N,d \geq 2$, every $\mathsf{A} \in GL_d(\mathbb{R})^N$ and every $s \in (0,d)$ the $\varphi^s$-equilibrium state of $\mathsf{A}$ is unique. This question was answered negatively by A. K\"aenm\"aki and M. Vippolainen in \cite{KaVi10} where an example with two ergodic $\varphi^s$-equilibrium states was constructed; K\"aenm\"aki and Vippolainen then asked whether the number of ergodic $\varphi^s$-equilibrium states is always finite. This question was answered affirmatively in two dimensions by D.-J. Feng and A. K\"aenm\"aki \cite{FeKa11}, in three dimensions by K\"aenm\"aki and the present author in \cite{KaMo16}, and in arbitrary dimensions by J. Bochi and the present author in \cite{BoMo17}, where the number of ergodic $\varphi^s$-equilibrium states was shown to be bounded by a number depending only on $d$ and $s$. It was shown in \cite{KaMo16} that the number of ergodic $\varphi^s$-equilibrium states can be at least as high as $(d-\lfloor s\rfloor){d \choose \lfloor s \rfloor} = \lceil s \rceil {d \choose \lceil s \rceil}$ when $s$ is noninteger, and at least ${d \choose s}$ when $s$ is an integer; in both cases the examples constructed were simultaneously triangularisable. In the integer case this lower bound can be seen to be sharp using the results of Feng and K\"aenm\"aki \cite{FeKa11}. On the other hand in the non-integer case the best available upper bound for the number of ergodic equilibrium states is ${d\choose  \lfloor s \rfloor}{d \choose\lceil s\rceil}$, proved in \cite{BoMo17}. The gap between these upper and lower bounds led to the following question of Bochi and the present author (\cite[Question 2]{BoMo17}): if $\mathsf{A} \in GL_d(\mathbb{R})^N$ and $s \in (0,d)\setminus \mathbb{Z}$, and $\mathsf{A}$ has at least $(d-\lfloor s\rfloor){d \choose \lfloor s\rfloor}$ ergodic $\varphi^s$-equilibrium states, does it follow that $\mathsf{A}$ is upper triangularisable and hence has precisely $(d-\lfloor s\rfloor){d \choose \lfloor s\rfloor}$ ergodic $\varphi^s$-equilibrium states? In this note we answer this question positively for $d\leq 4$, with the exception of the case $(d,s)=(4,2)$ where we have a sharp bound for the number of ergodic $\varphi^s$-equilibrium states but do not prove triangularisability. We prove:
\begin{theorem}\label{th:bound}
Let $\mathsf{A}=(A_1,\ldots,A_N)\in GL_d(\mathbb{R})^N$ and $0<s<d\leq 4$. Then the maximum possible number of ergodic $\varphi^s$-equilibrium states of $\mathsf{A}$ is precisely ${d \choose s}$ if $s$ is an integer and $(d-\lfloor s \rfloor){d \choose \lfloor s \rfloor} = \lceil s\rceil {d \choose \lceil s \rceil}$ otherwise. Moreover, if this maximum is attained and $(d,s) \neq (4,2)$, then $\mathsf{A}$ is simultaneously triangularisable. 
\end{theorem}
Theorem \ref{th:bound} is obtained as a consequence of the results of \cite{BoMo17} via a somewhat convoluted case-by-case analysis. Analogues of this argument in dimensions higher than 4 are complicated not only by the increasing number of sub-cases but also by a lack of sharp tools for treating those cases in which $\mathsf{A}$ is irreducible but some of its exterior powers are not. Indeed, it is precisely this issue which complicates our treatment of the case $d=4$, $s=2$: in that case our techniques lead easily to the conclusion that if ${4 \choose 2}$ ergodic $\varphi^2$-equilibrium states exist then $\mathsf{A}^{\wedge 2}$ is upper triangularisable, but to deduce from this that $\mathsf{A}$ is also upper triangularisable would require the application of nontrivial techniques from the theory of algebraic groups and Lie groups which we do not attempt to deploy here.

Let us now describe the second question which we address in this article.
Let $\mathsf{A}=(A_1,\ldots,A_N)\in GL_d(\mathbb{R})^N$ and let $\mu \in \mathcal{M}_\sigma$ be ergodic. The \emph{Lyapunov exponents} of $\mathsf{A}$ with respect to $\mu$ are defined to be the numbers
\[\Lambda_j (\mathsf{A},\mu)=\lim_{n \to \infty}  \frac{1}{n}\int \log \alpha_j(A_{i_n}\cdots A_{i_1})d\mu[(i_k)_{k=1}^\infty]\]
for $j=1,\ldots,d$, the existence of the limit being guaranteed by the subadditivity of the sequence
\begin{equation}\label{eq:etc}\int \log \left(\prod_{j=1}^\ell \alpha_j(A_{i_n}\cdots A_{i_1})\right)d\mu[(i_k)_{k=1}^\infty]\end{equation}
for each $\ell=1,\ldots,d$ (see below). 
In the article \cite{BaKaKo17}, B. B\'ar\'any, A. K\"aenm\"aki and H. Koivusalo asked the following question: if $\mu$ is a K\"aenm\"aki measure for the matrices $\mathsf{A}=(A_1,\ldots,A_N) \in GL_d(\mathbb{R})^N$, under what circumstances do the Lyapunov exponents of $\mathsf{A}$ with respect to $\mu$ take $d$ different values? This question was answered for K\"aenm\"aki measures of planar affine IFS by the author \cite[Theorem 13]{Mo17a}, but in higher dimensions the question seems more difficult to answer. In this note we give a partial answer by providing a simple checkable criterion for the separation of those Lyapunov exponents which are closest to the affinity dimension. We remark that for each $N,d \geq 2$ our criterion is satisfied for a dense, open, full-Lebesgue-measure set of tuples $(A_1,\ldots,A_N)\in GL_d(\mathbb{R})^N$.

If $1 \leq k \leq d$ we recall that the \emph{$k^{\mathrm{th}}$ exterior power} of $\mathbb{R}^d$ is 
the vector space spanned by formal expressions of the form $u_1\wedge \cdots \wedge u_k$ where $u_1,\ldots,u_k \in \mathbb{R}^d$, subject to the identifications 
\[(\lambda u_1)\wedge u_2 \wedge \cdots \wedge u_k= \lambda(u_1\wedge \cdots \wedge u_k),\]
 \[u_1 \wedge \cdots \wedge u_k = (-1)^{\mathrm{sign}(\varsigma)} u_{\varsigma(1)}\wedge \cdots \wedge u_{\varsigma(k)},\]
 \[(u_1 \wedge \cdots \wedge u_k) + (u_1' \wedge u_2 \wedge \cdots \wedge u_k) = (u_1+u_1')\wedge u_2 \cdots \wedge u_k\]
 where $\lambda \in \mathbb{R}$ and where $\varsigma \colon \{1,\ldots,k\} \to \{1,\ldots,k\}$ is any permutation. If an inner product $\langle \cdot,\cdot\rangle$ on $\mathbb{R}^d$ is understood, then
 \begin{equation}\label{eq:inducednorm}\langle u_1\wedge \cdots \wedge u_k,v_1\wedge \cdots \wedge v_k\rangle :=  \det [\langle u_i,v_j\rangle]_{i,j=1}^d\end{equation}
 extends by linearity to an inner product on $\wedge^k \mathbb{R}^d$. If $u_1,\ldots,u_d$ is a basis for $\mathbb{R}^d$ then the vectors $u_{i_1}\wedge \cdots \wedge u_{i_d}$ such that $1 \leq i_1 < i_2 < \cdots < i_k\leq d$ form a basis for $\wedge^k \mathbb{R}^d$, and in particular $\dim \wedge^k \mathbb{R}^d={d\choose k}$. If $A \colon \mathbb{R}^d \to \mathbb{R}^d$ is linear then we define $A^{\wedge k}$ to be the unique linear transformation of $\wedge^k \mathbb{R}^d$ such that $A^{\wedge k}(u_1\wedge \cdots \wedge u_k)=Au_1\wedge \cdots \wedge Au_k$ for all $u_1,\ldots,u_k \in \mathbb{R}^d$. We have $(A^{\wedge k})^T=(A^T)^{\wedge k}$ and $(AB)^{\wedge k}=A^{\wedge k}B^{\wedge k}$ for all linear endomorphisms $A,B$ of $\mathbb{R}^d$. If $A \colon \mathbb{R}^d \to \mathbb{R}^d$ is given and $e_1,\ldots,e_d$ is a basis for $\mathbb{R}^d$ given by (generalised) eigenvectors of $A$ then it is straightforward to check that vectors of the form $e_{i_1}\wedge \cdots \wedge e_{i_k}$ with $1 \leq i_1<\cdots <i_k\leq d$ are a basis for $\wedge^k\mathbb{R}^d$ given by (generalised) eigenvectors of $A^{\wedge k}$. It follows from these considerations that $\left\|A^{\wedge k}\|=\|(A^{\wedge k})^TA^{\wedge k}\right\|^{1/2}=\prod_{i=1}^k \alpha_i(A)$ for any linear endomorphism $A$ of $\mathbb{R}^d$ and any $1 \leq k \leq d$, where $\|\cdot\|$ denotes the Euclidean norm implied by the  inner product \eqref{eq:inducednorm}. This in particular implies the inequality $\prod_{i=1}^k \alpha_i(AB) \leq \left(\prod_{i=1}^k \alpha_i(A) \right)\left(\prod_{i=1}^k \alpha_i(B) \right)$ for all $A,B \in M_d(\mathbb{R})$ which guarantees the existence of the limit \eqref{eq:etc}. If $\mathsf{A}=(A_1,\ldots,A_N) \in GL_d(\mathbb{R})^N$ and $1 \leq k \leq d$ then we write $\mathsf{A}^{\wedge k}=(A_1^{\wedge k},\ldots,A_N^{\wedge k})$. We note the identity $\varphi^s(A)=\left\|A^{\wedge \lfloor s\rfloor}\right\|^{1+\lfloor s\rfloor -s} \left\|A^{\wedge \lceil s\rceil}\right\|^{s-\lfloor s \rfloor}$ which will be used extensively in this article.

We shall say that $\mathsf{A}=(A_1,\ldots,A_N) \in GL_d(\mathbb{R})^N$ is \emph{irreducible} if there is no proper nonzero subspace $V$ of $\mathbb{R}^d$ such that $A_iV \subseteq V$ for all $i=1,\ldots,N$, and we shall say that $\mathsf{A}$ is \emph{strongly irreducible} if there is no finite union $W=\bigcup_{j=1}^m V_j$ of proper nonzero subspaces $V \subset \mathbb{R}^d$ such that $A_iW \subseteq W$ for all $i=1,\ldots,N$. We will say that $\mathsf{A}$ is \emph{$k$-irreducible} (respectively \emph{$k$-strongly irreducible}) if $\mathsf{A}^{\wedge k}$ is irreducible (respectively strongly $k$-irreducible) in the same sense. For the purposes of this article we shall also say that $\mathsf{A}=(A_1,\ldots,A_N)$ is \emph{$k$-proximal} if there exist $i_1,\ldots,i_n \in \{1,\ldots,N\}$ such that the $k^{\mathrm{th}}$-largest and $(k+1)^{\mathrm{st}}$-largest eigenvalues of $A_{i_n}\cdots A_{i_1}$ have distinct absolute values. (This definition of $k$-proximality coincides with more standard notions of $k$-proximality if $\mathsf{A}$ is strongly $k$-irreducible -- see e.g. \cite[\S2]{GoGu96} -- but we shall find this terminology to be convenient for arbitrary $\mathsf{A}$.) We note that $\mathsf{A}$ is $k$-proximal if and only if $\mathsf{A}^{\wedge k}$ is $1$-proximal. 

In this note we prove the following theorem on the separation of Lyapunov exponents: 
\begin{theorem}\label{th:sep}
Let $\mathsf{A}=(A_1,\ldots,A_N) \in GL_d(\mathbb{R})^N$ and $0<s<d$. Then the following properties hold:
\begin{enumerate}[(i)]
\item
Suppose that $0 \leq k <s\leq k+1<d$ and that $\mathsf{A}$ is $\ell$-irreducible for $\ell=k$, $\ell=k+1$ and $\ell=k+2$. Suppose also that $\mathsf{A}$ is $(k+1)$-proximal. If additionally $\mathsf{A}$ is strongly $\ell$-irreducible either for $\ell=k$, or for both $\ell=k+1$ and $\ell=k+2$, then $\mathsf{A}$ has a unique $\varphi^s$-equilibrium state $\mu$, and $\Lambda_{k+1}(\mathsf{A},\mu)>\Lambda_{k+2}(\mathsf{A},\mu)$.
\item
Suppose that $0 < k \leq s<k+1\leq d$ and that $\mathsf{A}$ is $\ell$-irreducible for $\ell=k-1$, $\ell=k$ and $\ell=k+1$. Suppose also that $\mathsf{A}$ is $k$-proximal. If additionally $\mathsf{A}$ is strongly $\ell$-irreducible either for $\ell=k+1$, or for both $\ell=k-1$ and $\ell=k$, then $\mathsf{A}$ has a unique $\varphi^s$-equilibrium state $\mu$, and $\Lambda_{k}(\mathsf{A},\mu)>\Lambda_{k+1}(\mathsf{A},\mu)$.
\end{enumerate}
\end{theorem}
If $s$ is an integer then the irreducibility conditions may be very slightly weakened: see Remark 1 below. Our criterion is unfortunately insufficient to fully answer Barany, K\"aenm\"aki  and Koivusalo's question even in three dimensions, since for example we are not able to exclude the possibility that $\Lambda_1(\mathsf{A},\mu)=\Lambda_2(\mathsf{A},\mu)$ when $\mu$ is a $\varphi^s$-equilibrium state of $\mathsf{A} \in GL_3(\mathbb{R})^N$  and $2<s<3$, or $\Lambda_2(\mathsf{A},\mu)=\Lambda_3(\mathsf{A},\mu)$ when $\mu$ is a $\varphi^s$-equilibrium state of $\mathsf{A} \in GL_3(\mathbb{R})^N$ with $0<s<1$. In a sense these two cases are equivalent: see Remark 2 below.

The remainder of this article is structured as follows. In \S2 we present some general results and notations which will be applied in proving both of our main theorems. In \S3 and \S4 we present the proofs of Theorems \ref{th:bound} and \ref{th:sep} respectively, and in \S5 we examine the question of when a K\"aenm\"aki measure can be a Bernoulli measure and the implications for the Hausdorff dimension of self-affine measures.
\section{General preliminaries}

Let $N \geq 2$. We say that a \emph{word} over $\{1,\ldots,N\}$ is any finite sequence $\iii=(i_k)_{k=1}^n$ and let $\Sigma_N^*$ denote the set of all words over $\{1,\ldots,N\}$. If $\iii=(i_k)_{k=1}^n$ we say that $n$ is the \emph{length} of $\iii$ and define $|\iii|:=n$. If $\iii$ and $\jjj$ are elements of $\Sigma_N^*$ we define their concatenation $\iii\jjj$ to be the word of length $|\iii|+|\jjj|$ obtained by running first through the symbols of $\iii$ and then through the symbols of $\jjj$ in the obvious manner. If $A_1,\ldots,A_N \in M_d(\mathbb{R})$ and $\iii =(i_k)_{k=1}^n \in \Sigma_N^*$ we define $A_\iii:=A_{i_n}\cdots A_{i_1}$. We note that $A_\iii A_\jjj=A_{\jjj\iii}$ for all $\iii,\jjj \in \Sigma_N^*$. For all $x=(x_k)_{k=1}^\infty \in \Sigma_N$ we let $x|_n:=(x_k)_{k=1}^n \in \Sigma_N^*$. 

For the purposes of this article we shall say that a \emph{potential} is any function $\Phi \colon \Sigma_N^* \to (0,+\infty)$. We will say that a potential is \emph{submultiplicative} if $\Phi(\iii\jjj)\leq \Phi(\iii)\Phi(\jjj)$ for every $\iii,\jjj \in \Sigma_N^*$ and \emph{quasimultiplicative} if there exist a finite set $F \subset \Sigma_N^*$ and a real number $\delta>0$ such that $\max_{\kkk \in F}\Phi(\iii\kkk\jjj) \geq \delta \Phi(\iii)\Phi(\jjj)$ for every $\iii,\jjj \in \Sigma_N^*$. If $\Phi$ is a submultiplicative potential we define a sequence of functions $\Phi_n \colon \Sigma_N \to \mathbb{R}$ by $\Phi_n(x):=\Phi(x|_n)$, and observe that $\Phi_{n+m}(x)\leq \Phi_n(\sigma^mx)\Phi_m(x)$ for all $x \in \Sigma_N$ and $n,m\geq 1$. We define the \emph{asymptotic average} of a submultiplicative potential $\Phi$ with respect to an ergodic measure $\mu \in \mathcal{M}_\sigma$ to be the quantity
\[\Lambda(\Phi,\mu):=\lim_{n \to \infty}\frac{1}{n}\int \log \Phi_n d\mu =\inf_{n \geq1}\frac{1}{n}\int \log \Phi_n d\mu,\]
where we note that the existence of the limit follows by subadditivity. We define the \emph{pressure} of a submultiplicative potential $\Phi$ to be the quantity
\[P(\Phi):=\lim_{n \to \infty} \frac{1}{n}\log \sum_{|\iii|=n} \Phi(\iii)=\inf_{n \geq 1} \frac{1}{n}\log \sum_{|\iii|=n} \Phi(\iii)\]
which again is well-defined by subadditivity.  By the subadditive variational principle (see \cite{CaFeHu08}) we have
\begin{equation}\label{eq:savp}P(\Phi)=\sup_{\mu \in \mathcal{M}_\sigma} h(\mu)+\Lambda(\Phi,\mu)\end{equation}
and this supremum is always attained since $\mathcal{M}_\sigma$ is weak-* compact and $\mu \mapsto h(\mu)+\Lambda(\Phi,\mu)$ is upper semi-continuous.  We say that $\mu \in \mathcal{M}_\sigma$ is an \emph{equilibrium state} for a submultiplicative potential $\Phi$ if $P(\Phi)=h(\mu) + \Lambda(\Phi,\mu)$. If $\mathsf{A}=(A_1,\ldots,A_N)\in GL_d(\mathbb{R})^N$ is given, we will say (as in the introduction) that $\mu$ is a $\varphi^s$-equilibrium state of $\mathsf{A}$ if it is an equilibrium state of the potential $\Phi(\iii):=\varphi^s(A_\iii)$, and also that $\mu$ is a $\|\cdot\|^t$-equilibrium state of $\mathsf{A}$ if it is an equilibrium state of the potential $\Phi'(\iii):=\left\|A_\iii\right\|^t$, where $t>0$.

Our interest in subadditive potentials as a general class is motivated by the following special case of a theorem of D.-J. Feng (\cite[Theorem 5.5]{Fe11}):
\begin{proposition}[\cite{Fe11}]\label{pr:fe}
Let $N \geq 2$ and let $\Phi \colon \Sigma_N^* \to (0,+\infty)$ be a submultiplicative and quasimultiplicative potential. Then there exists a unique equilibrium state $\mu$ for $\Phi$, and moreover there exists $C>0$ depending only on $\Phi$ such that
\[C^{-1}e^{-|\iii|P(\Phi)} \Phi(\iii) \leq \mu([\iii]) \leq Ce^{-|\iii|P(\Phi)} \Phi(\iii) \]
for every $\iii \in \Sigma_N$.
\end{proposition}
The following property of $\varphi^s$-equilibrium states will be useful in both of the following two sections:
\begin{lemma}\label{le:5}
Let $\mathsf{A}=(A_1,\ldots,A_N) \in GL_d(\mathbb{R})^N$ and $0<s<d$. Then a measure $\mu$ is a $\varphi^s$-equilibrium state of $\mathsf{A}$ if and only if it is a $\varphi^{d-s}$-equilibrium state of $\mathsf{A}'=(A_1',\ldots,A_N')$, where for $i=1,\ldots,N$ we define
\[A_i':=|\det A_i|^{\frac{1}{d-s}} \left(A_i^{-1}\right)^T.\]
\end{lemma}
\begin{proof}
Define $\phi \colon GL_d(\mathbb{R}) \to GL_d(\mathbb{R})$ by $\phi(A):=|\det A|^{1/(d-s)} (A^{-1})^T$ and suppose that $k\leq s \leq k+1$ where $k$ is an integer. Then $\phi$ is a homomorphism and
\begin{align*}\varphi^{d-s}(\phi(A))&=|\det A| \varphi^{d-s}\left(A^{-1}\right)\\
&= |\det A| \alpha_1(A^{-1})\cdots \alpha_{d-k-1}(A^{-1}) \alpha_{d-k}(A^{-1})^{d-s-(d-k-1)}\\
&= |\det A| \alpha_d(A)^{-1}\cdots \alpha_{k+2}(A)^{-1} \alpha_{k+1}(A)^{s-k-1}\\
&= \alpha_1(A)\cdots \alpha_{k}(A) \alpha_{k+1}(A)^{s-k} =\varphi^s(A).\end{align*}
We deduce that the potentials $\Phi(\iii):=\varphi^s(A_\iii)$ and $\Phi'(\iii):=\varphi^{d-s}(A_\iii')$ satisfy
\[\Phi'(\iii)=\varphi^{d-s}(A_\iii')=\varphi^{d-s}(\phi(A_\iii))=\varphi^s(A_\iii)=\Phi(\iii)\]
for every $\iii \in \Sigma_N^*$, where the second equality exploits the fact that $\phi$ is a homomorphism. The result follows.
\end{proof}


\section{Proof of Theorem \ref{th:bound}}

The proof of Theorem \ref{th:bound} operates by appeal to a long series of lemmas. The following result may be easily deduced from the work of Feng and K\"aenm\"aki \cite{FeKa11} and is also a special case of \cite[Theorem 5]{BoMo17}:
\begin{lemma}\label{le:1}
Let $\mathsf{B}=(B_1,\ldots,B_N) \in GL_d(\mathbb{R})^N$ and $s>0$, and define a potential $\Phi \colon \Sigma_N^* \to \mathbb{R}$ by $\Phi(\iii):=\|B_\iii\|^s$. Then $\Phi$ has at most $d$ ergodic equilibrium states, and if exactly $d$ ergodic equilibrium states exist then $\mathsf{B}$ is simultaneously triangularisable. If on the other hand $\mathsf{B}$ is irreducible, then $\Phi$ is quasi-multiplicative and there is a unique equilibrium state for $\Phi$.
\end{lemma}
The following result was proved in \cite{KaMo16}:
\begin{lemma}[{\cite[Theorem 5]{KaMo16}}]\label{le:4}
Let $\mathsf{A}=(A_1,\ldots,A_N) \in GL_d(\mathbb{R})^N$ and $0<s<d$, and suppose that there exist $\ell \in \{1,\ldots,d-1\}$ and a basis for $\mathbb{R}^d$ in which we may write
\[A_i=\begin{pmatrix}B_i & C_i \\0&D_i\end{pmatrix}\]
where each $B_i$ has dimension $\ell \times \ell$ and each $D_i$ has dimension $(d-\ell)\times (d-\ell)$. Then the $\varphi^s$-equilibrium states of $\mathsf{A}$ are identical to the $\varphi^s$-equilibrium states of the tuple $\mathsf{A}' =(A_1',\ldots,A_N') \in GL_d(\mathbb{R})^N$ defined by
\[A_i':=\begin{pmatrix}B_i &0 \\0&D_i\end{pmatrix}.\]
\end{lemma}
Versions of the following principle are appealed to in a number of works such as \cite{BoMo17,FeKa11,KaMo16}:
\begin{lemma}\label{le:max}
Let $N \geq 2$ and let $\Phi \colon \Sigma_N^* \to (0,+\infty)$ be a submultiplicative potential. Suppose that there exist submultiplicative potentials $\Phi^1,\ldots,\Phi^m \colon \Sigma_N^* \to (0,+\infty)$ and a constant $C>0$ such that $C^{-1}\Phi(\iii)=\max_{1 \leq j \leq m}\Phi^j(\iii) \leq C\Phi(\iii)$ for every $\iii \in \Sigma_N^*$. If $\mu$ is an ergodic equilibrium state of $\Phi$, then it is an ergodic equilibrium state of at least one of the potentials $\Phi^j$.
\end{lemma} 
\begin{proof}
Clearly we have $P(\Phi) \geq P(\Phi^j)$ for each $j=1,\ldots,m$ by direct appeal to the definition of the pressure $P$. If $\mu$ is an ergodic equilibrium state for $\Phi$ then by the subadditive ergodic theorem we have for $\mu$-a.e. $x \in \Sigma_N$
\[\Lambda(\Phi,\mu)=\lim_{n \to \infty} \frac{1}{n}\log \Phi_n(x) = \lim_{n \to \infty}  \frac{1}{n}\log \max_{1 \leq j \leq m}\Phi^j_n(x) = \max_{1 \leq j \leq m} \lim_{n \to \infty}\frac{1}{n}\log \Phi^j_n(x) = \max_{1 \leq j \leq m}\Lambda(\Phi^j,\mu).\]
Let $j \in \{1,\ldots,m\}$ such that $\Lambda(\Phi,\mu)=\Lambda(\Phi^j,\mu)$. We have
\[P(\Phi)=h(\mu)+\Lambda(\Phi,\mu)=h(\mu)+\Lambda(\Phi^j,\mu)\leq P(\Phi^j)\leq P(\Phi)\]
by the subadditive variational principle and therefore $P(\Phi^j)=h(\mu)+\Lambda(\Phi^j,\mu)$ so that  $\mu$ is an equilibrium state of $\Phi^j$ as required.
\end{proof}
The following result is obtained from \cite[Theorem 5]{BoMo17} by taking $k=2$ and $n_1=1$:
\begin{lemma}\label{le:2}
Let $\mathsf{B}=(B_1,\ldots,B_N) \in GL_{d_1}(\mathbb{R})^N$ and $\mathsf{C}=(C_1,\ldots,C_N)\in GL_{d_2}(\mathbb{R})^N$ and define a potential $\Phi$ by $\Phi(\iii)=\|B_\iii\|^\beta \|C_\iii\|^\gamma$ for suitable real constants $\beta,\gamma>0$. Suppose that $\mathsf{B}$ is irreducible. Then the number of ergodic equilibrium states of $\Phi$ is not greater than $d_2$.
\end{lemma}
The following result recalls some arguments from \cite[\S7]{KaMo16}:
\begin{lemma}\label{le:3}
Suppose that $\mathsf{A}=(A_1,\ldots,A_N)\in GL_d(\mathbb{R})^N$ may be written as
\[X^{-1}A_iX = \begin{pmatrix}b_i&0\\0&C_i\end{pmatrix}\]
for each $i=1,\ldots,N$, where each $b_i$ is real, each $C_i$ is a $(d-1)\times (d-1)$ matrix, and $X \in GL_d(\mathbb{R})$. Let $s \in (1,d-1)$. Then every ergodic $\phi^s$-equilibrium state is either a $\varphi^s$-equilibrium state of $(C_1,\ldots,C_N)$, or a $\varphi^{s-1}$-equilibrium state of $(|b_1|^{1/(s-1)}C_1,\ldots,|b_N|^{1/(s-1)}C_N)$, or a $\|\cdot\|$-equilibrium state of $(|b_1|^{s-\lfloor s\rfloor}C_1^{\wedge \lfloor s \rfloor},\ldots,|b_N|^{s-\lfloor s \rfloor}C_N^{\wedge \lfloor s \rfloor})$. 
\end{lemma}
\begin{proof}
Let $\hat{A}_i:=X^{-1}AX$ for each $i=1,\ldots,N$. We have
\[ \|X^{-1}\|^{-1}\|X\|^{-1}\alpha_k(A_i) \leq \alpha_k(\hat{A}_i) \leq \|X^{-1}\|\cdot\|X\|\alpha_k(A_i)\]
for every $k=1,\ldots,d$ and $i=1,\ldots,N$, and it follows that $C^{-1}\varphi^s(A_\iii)\leq \varphi^s(\hat{A}_\iii) \leq C\varphi^s(A_\iii)$ for every $\iii \in \Sigma_N^*$ for some constant $C>0$ depending only on $X$. In particular $(A_1,\ldots,A_N)$ and $(\hat{A}_1,\ldots,\hat{A}_N)$ have the same $\varphi^s$-equilibrium states. 

The singular values of $\hat{A}_\iii$ are precisely $|b_\iii|$, $\alpha_1(C_\iii), \ldots, \alpha_d(C_\iii)$ in some order, with $\alpha_k(C_\iii)$ preceding $\alpha_{k+1}(C_\iii)$ for each $k=1,\ldots,d-1$. Let $\ell:=\lfloor s \rfloor$. We have for each $\iii \in \Sigma_N^*$
\begin{align*}\varphi^s(\hat{A}_\iii)&=\alpha_1(\hat{A}_\iii)\cdots \alpha_\ell(\hat{A}_\iii) \alpha_{\ell+1}(\hat{A}_\iii)^{s-\ell}\\
&=\max\left\{ \begin{array}{c}\alpha_1(C_\iii)\cdots \alpha_{\ell}(C_\iii)\alpha_{\ell+1}(C_\iii)^{s-\ell},\\
 |b_\iii| \alpha_1(C_\iii)\cdots \alpha_{\ell-1}(C_\iii)\alpha_\ell(C_\iii)^{s-\ell},\\
  \alpha_1(C_\iii)\cdots \alpha_{\ell-1}(C_\iii)\alpha_\ell(C_\iii)|b_\iii|^{s-\ell}\end{array}\right\}\\
&=\max\left\{\begin{array}{c}\varphi^s(C_\iii),\\\varphi^{s-1}(|b_\iii|^{1/(s-1)}C_\iii),\\\||b_i|^{s-\ell}C_\iii^{\wedge \ell}\|\end{array} \right\}\end{align*}
where the maximum is equal to the first term if $\alpha_{\ell+1}(C_\iii)\geq |b_\iii|$, the second if $\alpha_\ell(C_\iii)\leq |b_\iii|$, and the third if $\alpha_{\ell}(C_\iii)\leq |b_\iii| \leq \alpha_{\ell+1}(C_\iii)$.
Define $\Phi(\iii):=\varphi^s(A_\iii)$, $\Phi^1(\iii):=\varphi^s(C_\iii)$, $\Phi^2(\iii):=\varphi^{s-1}(|b_\iii|^{1/(s-1)} C_\iii)$, and $\Phi^3(\iii):=\||b_\iii|^{s-\ell} C_\iii^{\wedge \ell}\|$ for all $\iii \in \Sigma_N^*$. We have
\[C^{-1}\Phi(\iii) \leq \max_{1 \leq j \leq 3}\Phi^j(\iii) \leq  C\Phi(\iii)\]
for all $\iii \in \Sigma_N^*$ where $C>0$ is a suitable constant, and hence by Lemma \ref{le:max} every ergodic equilibrium state of $\Phi$ must be an ergodic equilibrium state of one of the three potentials $\Phi^j$. Thus every $\phi^s$-equilibrium state is either a $\varphi^s$-equilibrium state of $(C_1,\ldots,C_N)$, or a $\varphi^{s-1}$-equilibrium state of $(|b_1|C_1,\ldots,|b_N|C_N)$, or a $\|\cdot\|$-equilibrium state of $(|b_1|^{s-\lfloor s\rfloor}C_1^{\wedge \lfloor s \rfloor},\ldots,|b_N|^{s-\lfloor s \rfloor}C_N^{\wedge \lfloor s\rfloor})$ as claimed.
\end{proof}
The following result is a corollary of several results from \cite{KaMo16}:
\begin{lemma}\label{le:1.5}
Let $\mathsf{B}=(B_1,\ldots,B_N) \in GL_d(\mathbb{R})^N$ and $s \in (0,d) \setminus \mathbb{Z}$, and suppose that $\mathsf{B}$ is simultaneously triangularisable. Then there are at most $(d-\lfloor s\rfloor){d \choose \lfloor s \rfloor}$ ergodic $\varphi^s$-equilibrium states for $\mathsf{B}$.
\end{lemma}
\begin{proof}
By a change of basis we may assume that $\mathsf{B}$ is upper triangular; as in the proof of the previous lemma, this change of basis does not affect the $\varphi^s$-equilibrium states of $\mathsf{B}$. By repeated use of Lemma \ref{le:4} we see that the $\varphi^s$-equilibrium states of $\mathsf{B}$ are unchanged if the off-diagonal entries of each $B_i$ are deleted. It now follows by \cite[Theorem 4]{KaMo16} that the number of $\varphi^s$-equilibrium states can be at most $(d-\lfloor s\rfloor){d \choose \lfloor s \rfloor}$. 
\end{proof}

We may now prove Theorem \ref{th:bound}. We first claim that to prove the theorem it is sufficient to suppose that $s\leq\frac{d}{2}$. Indeed, suppose that this case of the theorem has been proved and that $\mathsf{A}=(A_1,\ldots,A_N)\in GL_d(\mathbb{R})^N$ with $\frac{d}{2}<s<d\leq 4$. By Lemma \ref{le:5} the number of $\varphi^s$-equilibrium states of $\mathsf{A}$ is equal to the number of $\varphi^{d-s}$-equilibrium states of $\mathsf{A}'=(A_1',\ldots,A_N')$ defined by $A_i':=|\det A_i|^{1/(d-s)}(A_i^{-1})^T$. If $s$ is noninteger then by hypothesis this is at most $(d-\lfloor d-s\rfloor){d \choose \lfloor d-s\rfloor} = \lceil s \rceil  {d\choose \lceil s\rceil}$ as required; if this maximum is attained then by hypothesis $\mathsf{A}'$ must be triangularisable, which clearly implies the triangularisability of $\mathsf{A}$. If $s$ is an integer then similarly the number of $\varphi^{d-s}$-equilibrium states is at most ${d\choose d-s}={d \choose s}$ and if this maximum is attained then by hypothesis $\mathsf{A}'$, and hence $\mathsf{A}$, is triangularisable. This proves the claim.

We now proceed to prove Theorem \ref{th:bound} successively for $d=2$, $d=3$ and $d=4$. If $d=2$ and $0<s \leq 1$ then $\varphi^s=\|\cdot\|^s$ and the result follows directly by Lemma \ref{le:1}, which in view of the previous claim completes the proof for $d=2$. If $d=3$ and $0<s \leq 1$ then similarly the result follows by Lemma \ref{le:1}. To complete the case $d=3$ we suppose that $1<s\leq \frac{3}{2}$. In this case we have $\varphi^s(A_\iii)=\left\|A_\iii\right\|^{2-s} \left\|A_\iii^{\wedge 2}\right\|^{s-1}$, so if $\mathsf{A}$ is irreducible then by Lemma \ref{le:2} the number of $\varphi^s$-equilibrium states is not greater than $3$, which is strictly less than the desired upper limit of $(3-\lfloor s\rfloor){3 \choose \lfloor s \rfloor}=6$. Suppose lastly that $\mathsf{A}$ is reducible. If $\mathsf{A}$ is triangularisable then the result follows by Lemma \ref{le:1.5}, so we suppose otherwise. If $\mathsf{A}$ has a $1$-dimensional invariant subspace then we may write
\[X^{-1}A_iX = \begin{pmatrix}b_i&D_i\\0&C_i\end{pmatrix}\]
for each $i=1,\ldots,N$, where each $b_i$ is real, each $D_i$ is a $1 \times 2$ matrix, $C_i$ is a $2 \times 2$ matrix, and $X \in GL_3(\mathbb{R})$. If instead it has a $2$-dimensional invariant subspace then we may write
\[X^{-1}A_iX = \begin{pmatrix}C_i&D_i\\0&b_i\end{pmatrix}\]
for each $i=1,\ldots,N$, where each $b_i$ is real, each $D_i$ is a $2 \times 1$ matrix, $C_i$ is a $2\times 2$ matrix, and $X \in GL_d(\mathbb{R})$. In either case $(C_1,\ldots,C_N)$ must be irreducible since otherwise $\mathsf{A}$ would be upper triangularisable, contradicting our assumption. Using Lemma \ref{le:4} and (in the second case only) a permutation of the  basis it follows that the $\varphi^s$-equilibrium states of $\mathsf{A}$ are precisely the $\varphi^s$-equilibrium states of $\mathsf{A}'$ where
\[A_i' := \begin{pmatrix}b_i&0\\0&C_i\end{pmatrix}\]
for each $i=1,\ldots,N$, where each $b_i$ is real and where the $2 \times 2$ matrices $(C_1,\ldots,C_N)$ are irreducible. Using Lemma \ref{le:3}, every ergodic $\varphi^s$-equilibrium state of $\mathsf{A}'$ is either a $\varphi^s$-equilibrium state of $(C_1,\ldots,C_N)$, a $\varphi^{s-1}$-equilibrium state of $(|b_1|C_1,\ldots,|b_N|C_N)$, or a $\|\cdot\|$-equilibrium state of $(|b_1|^{s-1}C_1,\ldots,|b_N|^{s-1}C_N)$. By appeal to the case $d=2$ and the fact that the matrices $C_i$ are not simultaneously triangularisable there can be at most one equilibrium state of the first type; by the same principle, there can be at most one equilibrium state of the second type; and by Lemma \ref{le:1} there can be at most one equilibrium state of the third type. We have shown that if  $d=3$, $1<s\leq \frac{3}{2}$ and $\mathsf{A}$ is not upper triangularisable then no more than three ergodic equilibrium states can exist, and this completes the proof of the theorem in the case $d=3$.

We now consider the case $d=4$. If $s=2$ then the result follows by noting that $\varphi^2(A_\iii)=\left\|A_\iii^{\wedge 2}\right\|$ and appealing to Lemma \ref{le:1}. If $0<s \leq 1$ then the result follows from Lemma \ref{le:1} as before. By our initial claim it remains only to consider the case $s \in (1,2)$. If $\mathsf{A}$ is triangularisable then the result follows from Lemma \ref{le:1.5} as in the case $d=3$, so it remains only to show that there are strictly fewer than $(4-1){4 \choose 1}=12$ ergodic $\varphi^s$-equilibrium states when $1<s<2$ and $\mathsf{A}$ is not triangularisable. If $\mathsf{A}$ is irreducible then by Lemma \ref{le:2} there are not more than $4$ ergodic $\varphi^s$-equilibrium states for $\mathsf{A}$. We therefore assume for the remainder of the proof that $\mathsf{A}$ is reducible but not triangularisable.

If $\mathsf{A}$ has a $1$-dimensional or $1$-codimensional invariant subspace then by changing basis, eliminating off-diagonal blocks and changing the basis once more we may as in the case $d=3$ reduce to the problem of finding the $\varphi^s$-equilibrium states of $\mathsf{A}'=(A_1',\ldots,A_N')$ where
\[A_i' := \begin{pmatrix}b_i&0\\0&C_i\end{pmatrix}\]
for each $i=1,\ldots,N$, each $b_i$ is real, each $C_i$ has dimension $3 \times 3$ and $(C_1,\ldots,C_N)$ is not simultaneously triangularisable. By Lemma \ref{le:3} every ergodic $\varphi^s$-equilibrium state of $\mathsf{A}'$ is either a $\varphi^s$-equilibrium state of $(C_1,\ldots,C_N)$, a $\varphi^{s-1}$-equilibrium state of $(|b_1|C_1,\ldots,|b_N|C_N)$, or a $\|\cdot\|$-equilibrium state of $(|b_1|^{s-1}C_1,\ldots,|b_N|^{s-1}C_N)$. By appeal to the case $d=3$ there must be fewer than six equilibrium states of the first type, and by appeal to Lemma \ref{le:1} there can be no more than two equilibrium states each of the second and third types. In particular the number of ergodic $\varphi^s$-equilibrium states of $\mathsf{A}$ must be less than ten when a $1$-dimensional or $1$-codimensional invariant subspace exists but $\mathsf{A}$ is not triangularisable.

The final remaining case is that in which $\mathsf{A}$ has a $2$-dimensional invariant subspace but no $1$-dimensional or $1$-codimensional invariant subspace. By a suitable change of basis we may write
\[X^{-1}A_iX = \begin{pmatrix}B_i&D_i\\0&C_i\end{pmatrix}\]
for each $i=1,\ldots,N$ where each $B_i$, $C_i$ and $D_i$ is a $2\times 2$ real matrix and each of the tuples $(B_1,\ldots,B_N)$ and $(C_1,\ldots,C_N)$ is irreducible. Using Lemma \ref{le:4} it follows that the $\varphi^s$-equilibrium states of $\mathsf{A}$ are precisely the $\varphi^s$-equilibrium states of  $\mathsf{A}'=(A_1',\ldots,A_N')$ where
\[A_i' := \begin{pmatrix}B_i&0\\0&C_i\end{pmatrix}\]
for each $i=1,\ldots,N$. We note that for each $\iii \in \Sigma_N^*$ the four singular values of $A_\iii'$ are precisely $\alpha_1(B_\iii)$, $\alpha_2(B_\iii)$, $\alpha_1(C_\iii)$ and $\alpha_2(C_\iii)$ in some order, with $\alpha_1(B_\iii)$ preceding $\alpha_2(B_\iii)$ and $\alpha_1(C_\iii)$ preceding $\alpha_2(C_\iii)$. In particular if we define four potentials by $\Phi^1(\iii):=\alpha_1(B_\iii)\alpha_2(B_\iii)^{s-1}=\varphi^s(B_\iii)$, $\Phi^2(\iii):=\alpha_1(B_\iii)\alpha_1(C_\iii)^{s-1}=\|B_\iii\|\cdot\|C_\iii\|^{s-1}$, $\Phi^3(\iii):=\alpha_1(B_\iii)^{s-1}\alpha_1(C_\iii)=\|B_\iii\|^{s-1}\|C_\iii\|$, $\Phi^4(\iii):=\alpha_1(C_\iii)\alpha_2(C_\iii)^{s-1}=\varphi^s(C_\iii)$, then 
\[\varphi^s(A_\iii')=\max_{1 \leq j \leq 4}\Phi^j(\iii)\]
for every $\iii \in \Sigma_N^*$. It follows by Lemma \ref{le:max} that if $\mu$ is an ergodic $\varphi^s$-equilibrium state of $\mathsf{A}$ then it is an equilibrium state of one of the potentials $\Phi^j$. By appeal to the case $d=2$ and the irreducibility of $(B_1,\ldots,B_N)$ and $(C_1,\ldots,C_N)$ the potentials $\Phi^1$ and $\Phi^4$ can contribute at most one ergodic equilibrium state each, and by appeal to Lemma \ref{le:2} and irreducibility the potentials $\Phi^2$ and $\Phi^3$ can contribute at most two ergodic equilibrium states each. In particular the number of ergodic $\varphi^s$-equilibrium states of $\mathsf{A}$ in this case is not higher than six.  The proof of the theorem is complete.


\section{Separation of Lyapunov exponents}
The following result is a special case of \cite[Corollary 2.2]{BoMo17}:
\begin{lemma}\label{le:2parts}
Let $\mathsf{C}:=(C_1,\ldots,C_N) \in GL_{d_1}(\mathbb{R})^N$ and $\mathsf{D}:=(D_1,\ldots,D_N) \in GL_{d_2}(\mathbb{R})^N$, and let $\gamma,\delta>0$. Define a submultiplicative potential $\Phi \colon \Sigma_N^* \to (0,+\infty)$ by $\Phi(\iii):=\|C_\iii\|^\gamma \|D_\iii\|^\delta$. If $\mathsf{C}$ and $\mathsf{D}$ are both irreducible, and at least one of them is strongly irreducible, then $\Phi$ is quasimultiplicative.
\end{lemma}
The proof of Theorem \ref{th:sep} rests on the following simple lemma.
\begin{lemma}\label{le:mma}
Let $\Phi^1,\Phi^2 \colon \Sigma_N^* \to \mathbb{R}$ be sub-multiplicative and quasi-multiplicative potentials, and let $\mu$ be the unique equilibrium state of $\Phi^1$. Suppose that $\Phi^1(\iii) \geq \Phi^2(\iii)$ for every $\iii \in \Sigma_N^*$ and that $\Lambda(\Phi^1,\mu)=\Lambda(\Phi^2,\mu)$. Then there exists $C>0$ such that $C^{-1}\Phi^1(\iii) \leq \Phi^2(\iii) \leq C\Phi^1(\iii)$ for every $\iii \in \Sigma_N^*$.
\end{lemma}
\begin{proof}
By Proposition \ref{pr:fe} each of $\Phi^1$ and $\Phi^2$ has a unique equilibrium state.
Since $\Phi^1 \geq \Phi^2$ it is clear from the definition of the pressure that $P(\Phi^1)\geq P(\Phi^2)$. We deduce
\[P(\Phi^2) \geq h(\mu)+\Lambda(\Phi^2,\mu) = h(\mu)+\Lambda(\Phi^1,\mu) = P(\Phi^1)\geq P(\Phi^2)\]
using the subadditive variational principle \eqref{eq:savp} and the hypothesis $\Lambda(\Phi^1,\mu)=\Lambda(\Phi^2,\mu)$. Hence $\mu$ is also the unique equilibrium state of $\Phi^2$ and therefore by Proposition \ref{pr:fe} there exist $C_1,C_2>0$ such that
\[C_1^{-1}\Phi^1(\iii) \leq \mu([\iii]) \leq C_1\Phi^1(\iii)\]
and
\[C_2^{-1}\Phi^2(\iii) \leq \mu([\iii]) \leq C_2\Phi^2(\iii)\]
for all $\iii \in \Sigma_N^*$. The result follows with $C:=C_1C_2$.
\end{proof}
\begin{proof}[Proof of Theorem \ref{th:sep}]
To prove (i) we take $\Phi^1(\iii):=\varphi^s(A_\iii)=\|A^{\wedge k}_\iii\|^{k+1-s}\|A^{\wedge (k+1)}_\iii\|^{s-k}$ and $\Phi^2(\iii):=\|A^{\wedge k}_\iii \|^{1+\frac{k-s}{2}} \|A^{\wedge(k+2)}_\iii\|^{\frac{s-k}{2}}$. By Lemma \ref{le:2parts} each of these two potentials is submultiplicative and quasimultiplicative. Suppose for a contradiction that $\Lambda_{k+1}(\mathsf{A},\mu)=\Lambda_{k+2}(\mathsf{A},\mu)$. We have
\begin{align*}\Phi^1(\iii)&=\left\|A^{\wedge k}_\iii\right\|^{k+1-s}\left\|A^{\wedge (k+1)}_\iii\right\|^{s-k}\\
& =\alpha_1(A_\iii)\cdots \alpha_k(A_\iii)\alpha_{k+1}(A_\iii)^{s-k}\\
&\geq \alpha_1(A_\iii)\cdots \alpha_k(A_\iii)\alpha_{k+1}(A_\iii)^{\frac{s-k}{2}} \alpha_{k+2}(A_\iii)^{\frac{s-k}{2}}\\
&=  \left\|A^{\wedge k}_\iii \right\|^{1+\frac{k-s}{2}} \left\|A^{\wedge(k+2)}_\iii\right\|^{\frac{s-k}{2}}=\Phi^2(\iii)\end{align*}
where the middle inequality follows from $\alpha_{k+1}(A_\iii)\geq \alpha_{k+2}(A_\iii)$. On the other hand
\begin{align*}\Lambda(\Phi^1,\mu) &= \sum_{i=1}^{k} \Lambda_i(\mathsf{A},\mu) +(s-k)\Lambda_{k+1}(\mathsf{A},\mu)\\
&= \sum_{i=1}^{k} \Lambda_i(\mathsf{A},\mu) + \left(\frac{s-k}{2}\right)\Lambda_{k+1}(\mathsf{A},\mu) +\left(\frac{s-k}{2}\right)\Lambda_{k+2}(\mathsf{A},\mu)=\Lambda(\Phi^2,\mu)\end{align*}
by hypothesis. It follows by Lemma \ref{le:mma} that there exists $C>0$ such that
\[C^{-1} \leq \frac{\left\|A^{\wedge k}_\iii \right\|^{k+1-s} \left\|A^{\wedge(k+1)}_\iii\right\|^{s-k}}{\left\|A^{\wedge k}_\iii \right\|^{1+\frac{k-s}{2}} \left\|A^{\wedge(k+2)}_\iii\right\|^{\frac{s-k}{2}}}\leq C \]
for every $\iii \in \Sigma_N^*$, and this simplifies to
\[C^{-1} \leq \frac{\alpha_{k+1}(A_\iii)^{\frac{s-k}{2}}}{\alpha_{k+2}(A_\iii)^{\frac{s-k}{2}}} \leq C.\]
It follows by Yamamoto's Theorem that for every $ \iii \in \Sigma_N^*$
\[\frac{\lambda_{k+1}(A_\iii)}{\lambda_{k+2}(A_\iii)}=\lim_{n \to \infty}  \left(\frac{\alpha_{k+1}(A_\iii^n)}{\alpha_{k+2}(A_\iii^n)}\right)^{\frac{1}{n}}=1,\]
where  $\lambda_i(B)$ denotes the absolute value of the $i^{\mathrm{th}}$-largest eigenvalue of the matrix $B$. This  
contradicts the hypothesis that $\mathsf{A}$ is $(k+1)$-proximal, and we conclude that $\Lambda_{k+1}(\mathsf{A},\mu)>\Lambda_{k+2}(\mathsf{A},\mu)$ as required.

The proof of (ii) is similar. In this case we take $\Phi^1(\iii):=\varphi^s(A_\iii)=\|A^{\wedge k}_\iii\|^{k+1-s}\|A^{\wedge (k+1)}_\iii\|^{s-k}$ and $\Phi^2(\iii):=\|A^{\wedge (k-1)}_\iii \|^{\frac{1-s+k}{2}} \|A^{\wedge(k+1)}_\iii\|^{\frac{1+s-k}{2}}$, and again by Lemma \ref{le:2parts} each of these two potentials is submultiplicative and quasimultiplicative. Assuming for a contradiction that  $\Lambda_{k}(\mathsf{A},\mu)=\Lambda_{k+1}(\mathsf{A},\mu)$, we note that
\begin{align*}\Phi^1(\iii)&=\left\|A^{\wedge k}_\iii\right\|^{k+1-s}\left\|A^{\wedge (k+1)}_\iii\right\|^{s-k} \\
&=\alpha_1(A_\iii)\cdots \alpha_k(A_\iii)\alpha_{k+1}(A_\iii)^{s-k}\\
&\geq \alpha_1(A_\iii)\cdots \alpha_{k-1}(A_\iii)\alpha_{k}(A_\iii)^{\frac{1+s-k}{2}} \alpha_{k+1}(A_\iii)^{\frac{1+s-k}{2}}\\
& = \left\|A^{\wedge (k-1)}_\iii \right\|^{\frac{1-s+k}{2}} \left\|A^{\wedge(k+1)}_\iii\right\|^{\frac{1+s-k}{2}}=\Phi^2(\iii)\end{align*}
where the middle inequality follows from $\alpha_{k}(A_\iii)\geq \alpha_{k+1}(A_\iii)$, and also 
\begin{align*}\Lambda(\Phi^1,\mu) &= \sum_{i=1}^{k-1} \Lambda_i(\mathsf{A},\mu) +\Lambda_k(\mathsf{A},\mu)+(s-k)\Lambda_{k+1}(\mathsf{A},\mu)\\
&= \sum_{i=1}^{k-1} \Lambda_i(\mathsf{A},\mu) + \left(\frac{1+s-k}{2}\right)\Lambda_{k}(\mathsf{A},\mu) +\left(\frac{1+s-k}{2}\right)\Lambda_{k+1}(\mathsf{A},\mu)=\Lambda(\Phi^2,\mu).\end{align*}
Hence by Lemma \ref{le:mma} that there exists $C>0$ such that
\[C^{-1} \leq \frac{\left\|A^{\wedge k}_\iii \right\|^{k+1-s} \left\|A^{\wedge(k+1)}_\iii\right\|^{s-k}}{\left\|A^{\wedge (k-1)}_\iii \right\|^{\frac{1-s+k}{2}} \left\|A^{\wedge(k+1)}_\iii\right\|^{\frac{1+s-k}{2}}}\leq C \]
and consequently
\[C^{-1} \leq \frac{\alpha_{k}(A_\iii)^{\frac{1+k-s}{2}}}{\alpha_{k+1}(A_\iii)^{\frac{1+k-s}{2}}} \leq C\]
for every $\iii \in \Sigma_N^*$. We likewise deduce the contradiction $\lambda_k(A_\iii)=\lambda_{k+1}(A_\iii)$ for every $\iii \in \Sigma_N^*$.
\end{proof}
\emph{Remark 1.} In the case where $k$ is an integer the irreducibility conditions on $\mathsf{A}$ may be relaxed slightly. The role of these conditions in the proof is to ensure that both $\Phi^1$ and $\Phi^2$ are quasi-multiplicative; in (i), if $s=k+1$ then $\Phi^1(\iii)$ reduces to $\left\|A_\iii\right\|^{\wedge(k+1)}$, so it suffices to assume that $\mathsf{A}^{\wedge \ell}$ is irreducible for $\ell \in \{k,k+1,k+2\}$ and strongly irreducible for either $\ell=k$ or $\ell=k+2$. Equivalently, if $s=k$ in (ii) then $\Phi^1(\iii)=\left\|A_\iii\right\|^{\wedge k}$ and we may assume that $\mathsf{A}^{\wedge \ell}$ is irreducible for $\ell \in \{k-1,k,k+1\}$ and strongly irreducible for one of $\ell=k-1$ and $\ell=k+1$.

\emph{Remark 2.} We observe that statements (i) and (ii) of Theorem \ref{th:sep} are in fact equivalent to one another via Lemma \ref{le:5}; we leave the details to the reader.

\section{When is a K\"aenm\"aki measure a Bernoulli measure?}

Suppose that $T_1,\ldots,T_N \colon \mathbb{R}^d \to \mathbb{R}^d$ are affine contractions with linear parts $A_1,\ldots,A_N \in GL_d(\mathbb{R})^N$ having affinity dimension $s \in (0,d)$, and let $X= \bigcup_{i=1}^N T_iX$ be their attractor. Recall that a \emph{self-affine measure} with respect to $T_1,\ldots,T_N$ is a Borel probability measure $m$ on $\mathbb{R}^d$ such that $\sum_{i=1}^N p_i m(T_i^{-1}B)=m(B)$ for all Borel sets $B \subset \mathbb{R}^d$ and for some fixed probability vector $(p_1,\ldots,p_N)$. If we define $\pi \colon \Sigma_N \to X$ by
\[\pi\left[(x_k)_{k=1}^\infty\right]:=\lim_{n \to \infty} T_{x_1}\cdots T_{x_n}v\]
for all $v \in \mathbb{R}^d$ then this limit exists, belongs to $X$ and is independent of $v$. A measure on $X$ is then self-affine if and only if it is the projection via $\pi$ of a Bernoulli measure on $\Sigma_N$.

Let $\rho(A)$ denote the spectral radius of the matrix $A$. It was shown by J.E. Hutchinson in \cite{Hu81} that if the affinities $T_i$ are all similarities -- that is, if $\rho(A_i)^{-1}A_i \in O(d)$ for every $i=1,\ldots,N$ -- and if the affinities $T_i$ satisfy the Open Set Condition, then there exists a self-affine measure on the attractor $X$ whose Hausdorff dimension is equal to $s$. We note the following partial converse to this statement in two dimensions:
\begin{proposition}\label{pr:onoob}
Let $T_1,\ldots,T_N \colon \mathbb{R}^2 \to \mathbb{R}^2$ be invertible affine contractions with linear parts $A_1,\ldots,A_N \in GL_2(\mathbb{R})$, and suppose that $(A_1,\ldots,A_N)$ is irreducible and that $\dimaff (A_1,\ldots,A_N)\in(0,2)$. If there exists a self-affine measure on the attractor $X=\bigcup_{i=1}^N T_iX$ with Hausdorff dimension equal to $s:=\dimaff (A_1,\ldots,A_N)$, then $T_1,\ldots,T_N$ are similitudes with respect to some inner product on $\mathbb{R}^2$.
\end{proposition}
\begin{proof}
The proof reprises parts of \cite[\S5]{Mo17a}; we include these parts in order to better illuminate the problem which follows. If $\mu \in \mathcal{M}_\sigma$ then one may show that $\dim_H \pi_*\mu \leq s$ with equality only if $\mu$ is a $\varphi^s$-equilibrium state of $(A_1,\ldots,A_N)$, see \cite{Ka04b}. Let $m=\pi_*\mu$ be the hypothesised self-affine measure of Hausdorff dimension $s$, so that $\mu$ is a $\varphi^s$-equilibrium state of $(A_1,\ldots,A_N)$ which is a Bernoulli measure. Suppose firstly that $s \leq 1$ and therefore $\varphi^s(A_\iii)=\|A_\iii\|^s$ for every $\iii \in \Sigma_N^*$. The measure $\mu$ then satisfies
\[C^{-1}\|A_\iii\|^s \leq e^{|\iii| P(\mathsf{A},\varphi^s)} \mu([\iii]) \leq C\|A_\iii\|^s\]
for every $\iii \in \Sigma_N$ by the combination of Lemma \ref{le:1} and Proposition \ref{pr:fe}. Let $\iii,\jjj \in \Sigma_N$ be arbitrary; then we deduce
\[C^{-1}\|A_\iii^n\|^s \leq e^{n|\iii| P(\mathsf{A},\varphi^s)} \mu([\iii^n])=e^{n|\iii| P(\mathsf{A},\varphi^s)} \mu([\iii])^n \leq C\|A_\iii^n\|^s,\]
\[C^{-1}\|A_\jjj^n\|^s \leq e^{n|\jjj| P(\mathsf{A},\varphi^s)} \mu([\jjj^n])=e^{n|\jjj| P(\mathsf{A},\varphi^s)} \mu([\jjj])^n \leq C\|A_\jjj^n\|^s\]
and
\[C^{-1}\left\|(A_\iii A_\jjj)^n\right\|^s \leq e^{n(|\iii|+|\jjj|) P(\mathsf{A},\varphi^s)} \mu([(\jjj\iii)^n])=e^{n(|\iii|+|\jjj|) P(\mathsf{A},\varphi^s)} \mu([\iii])^n\mu([\jjj])^n \leq C\|(A_\iii A_\jjj)^n\|^s\]
for every $n \geq 1$, where $\kkk^n$ refers to the word formed by concatenating $n$ successive copies of $\kkk$ and where we have used the fact that $\mu$ is a Bernoulli measure. In particular
\[C^{-3/s} \|A_\iii^n\|\cdot\|A_\jjj^n\| \leq \left\|\left(A_\iii A_\jjj\right)^n\right\| \leq C^{3/s} \|A_\iii^n\|\cdot\|A_\jjj^n\|\]
for every $n \geq 1$ and $\iii,\jjj \in \Sigma_N$ so that by Gelfand's formula
\[\rho(A_\iii A_\jjj)=\rho(A_\iii)\rho(A_\jjj)\]
for all $\iii,\jjj \in \Sigma_N$. Thus the semigroup $\Gamma:=\{A_\iii \colon \iii \in \Sigma_N^*\}\subset GL_2(\mathbb{R})$ is irreducible and has the property that $\rho \colon \Gamma \to \mathbb{R}$ is multiplicative; but by a theorem of Protasov and Voynov \cite[Theorem 2]{PrVo17} this implies that there exists $B \in GL_2(\mathbb{R})$ such that $\rho(A)^{-1}B^{-1}AB \in O(2)$ for all $A \in \Gamma$. It follows in particular that for all $i=1,\ldots,N$ the matrix $\rho(A_i)^{-1}A_i$ is an isometry with respect to the inner product $(u,v)\mapsto \langle Bu,Bv\rangle$, and therefore each $T_i$ is a similitude with respect to that same inner product. This proves the proposition in the case $s \leq 1$. In the case $1 \leq s<2$ we similarly obtain $\rho(A_\iii A_\jjj)\equiv \rho(A_\iii)\rho(A_\jjj)$ by using $\varphi^s(A_\iii)=|\det A_\iii|^{s-1} \|A_\iii\|^{2-s}$ in place of $\|A_\iii\|^s$ throughout.
\end{proof}
It is natural to ask whether the above result may be extended beyond the planar case. We make the following conjecture:
\begin{conjecture}\label{co:xswain}
Let $T_1,\ldots,T_N \to \mathbb{R}^d \to \mathbb{R}^d$ be invertible affine contractions with linear parts $A_1,\ldots,A_N \in GL_d(\mathbb{R})$, and suppose that $s:=\dimaff (A_1,\ldots,A_N)\in (0,d)$. Suppose additionally that $(A_1^{\wedge \lfloor s \rfloor},\ldots,A_N^{\wedge \lfloor s\rfloor})$ and $(A_1^{\wedge \lceil s \rceil},\ldots,A_N^{\wedge \lceil s\rceil})$ are both irreducible and that at least one of them is strongly irreducible. If there exists a self-affine measure on the attractor $X=\bigcup_{i=1}^N T_iX$ with Hausdorff dimension equal to $s:=\dimaff (A_1,\ldots,A_N)$, then $T_1,\ldots,T_N$ are similitudes with respect to some inner product on $\mathbb{R}^d$.
\end{conjecture}
Conjecture \ref{co:xswain} may be restated in terms of K\"aenm\"aki measures as follows: if $A_1,\ldots,A_N \in GL_d(\mathbb{R})$ are contractions with affinity dimension $s \in (0,d)$, 
and $(A_1^{\wedge \lfloor s \rfloor},\ldots,A_N^{\wedge \lfloor s\rfloor})$ and $(A_1^{\wedge \lceil s \rceil},\ldots,A_N^{\wedge \lceil s\rceil})$ are both irreducible with at least one being strongly irreducible, then either the $A_i$'s are simultaneously conjugate to similarity matrices or their (unique) K\"aenm\"aki measure is not a Bernoulli measure.

The irreducibility hypothesis on the exterior powers of the matrices $A_i$ implies using \cite[Theorem 3]{KaMo16} that there is a unique $\varphi^s$-equilibrium state for $(A_1,\ldots,A_N)$ which has the Gibbs property
\[C^{-1}\varphi^s(A_\iii) \leq e^{|\iii| P(\mathsf{A},\varphi^s)} \mu([\iii]) \leq C\varphi^s(A_\iii)\]
for all $\iii \in \Sigma_N^*$. By following the argument of Proposition \ref{pr:onoob} we find that the inequality
\[C^{-3} \varphi^s(A_\iii^n) \varphi^s(A_\jjj^n) \leq \varphi^s\left(\left(A_\iii A_\jjj\right)^n\right) \leq C^3 \varphi^s(A_\iii^n) \varphi^s(A_\jjj^n)\]
is satisfied for all $\iii,\jjj \in \Sigma_N^*$ and $n \geq 1$. Let $k:=\lfloor s\rfloor$. Using Gelfand's formula together with the identity $\varphi^s(A)=\|A^{\wedge \lfloor s \rfloor}\|^{1+\lfloor s \rfloor-s} \|A^{\wedge \lceil s\rceil }\|^{s-\lfloor s\rfloor} $ it follows that 
\begin{equation}\label{eq:qe}\rho((A_\iii A_\jjj)^{\wedge \lfloor s\rfloor})^{1+\lfloor s\rfloor -s} \rho((A_\iii A_\jjj)^{\wedge \lceil s\rceil})^{s-\lfloor s\rfloor }  = \rho(A_\iii^{\wedge \lfloor s\rfloor})^{1+\lfloor s\rfloor -s} \rho(A_\iii^{\wedge \lceil s\rceil})^{s-\lfloor s\rfloor }  \rho(A_\jjj^{\wedge \lfloor s\rfloor})^{1+\lfloor s\rfloor -s} \rho(A_\jjj^{\wedge \lceil s\rceil})^{s-\lfloor s\rfloor }  \end{equation}
for all $\iii,\jjj \in \Sigma_N^*$.
In order to prove Conjecture \ref{co:xswain} by the method of Proposition \ref{pr:onoob} we would need to know under what circumstances the equation \eqref{eq:qe} implies that the matrices $\rho(A_\iii)^{-1}A_\iii$ belong to a bounded subgroup of $GL_d(\mathbb{R})$. We note that by replacing $A_i$ with $(\rho(A_i^{\wedge \lfloor s\rfloor})^{1+\lfloor s\rfloor -s} \rho(A_i^{\wedge \lceil s\rceil})^{s-\lfloor s\rfloor })^{-1/s}A_i$ for each $i=1,\ldots,N$ this problem may be reduced to the study of semigroups $\Gamma \subset GL_d(\mathbb{R})$ with the property that
\begin{equation}\label{eq:fw}
\rho\left(A^{\wedge \lfloor s\rfloor}\right)^{1+\lfloor s\rfloor -s} \rho\left(A^{\wedge \lceil s\rceil}\right)^{s-\lfloor s\rfloor }  = 1 
\end{equation}
for all $A \in \Gamma$. The structure of such semigroups is at the present time opaque. We therefore ask the following questions in support of Conjecture \ref{co:xswain}:
\begin{question}\label{qu:uq}
Let $\Gamma \subset GL_d(\mathbb{R})$ be a semigroup and let $s \in (0,d)$, and suppose that the equation \eqref{eq:fw} is satisfied for all $A \in \Gamma$. Assume furthermore that the sets $\{A^{\wedge \lfloor s \rfloor} \colon A \in \Gamma\}$ and $\{A^{\wedge \lceil s \rceil}\colon A \in \Gamma\}$ are strongly irreducible. Does it follow that $\Gamma$ is contained in a bounded subgroup of $GL_d(\mathbb{R})$?
\end{question}
We note that some degree of irreducibility must be assumed in the above question since otherwise counterexamples consisting only of diagonal matrices may be constructed. Such examples suggest the following question:
\begin{question}
Let $\Gamma \subset GL_d(\mathbb{R})$  be a semigroup and let $s \in (0,d)$, and suppose that
\[\rho((AB)^{\wedge \lfloor s\rfloor})^{1+\lfloor s\rfloor -s} \rho((AB)^{\wedge \lceil s\rceil})^{s-\lfloor s\rfloor }  = \rho(A^{\wedge \lfloor s\rfloor})^{1+\lfloor s\rfloor -s} \rho(A^{\wedge \lceil s\rceil})^{s-\lfloor s\rfloor }  \rho(B^{\wedge \lfloor s\rfloor})^{1+\lfloor s\rfloor -s} \rho(B^{\wedge \lceil s\rceil})^{s-\lfloor s\rfloor }  \]
for all $A,B \in \Gamma$. Does it follow that in fact $\rho((AB)^{\wedge \lfloor s\rfloor})=\rho(A^{\wedge \lfloor s\rfloor})\rho(B^{\wedge \lfloor s\rfloor})$ and $\rho((AB)^{\wedge \lceil s\rceil})=\rho(A^{\wedge \lceil s\rceil})\rho(B^{\wedge \lceil s\rceil})$ for all $A,B \in \Gamma$?
\end{question}
By the conventions of algebraic geometry a function $\phi \colon GL_d(\mathbb{R}) \to \mathbb{R}$ is called a \emph{polynomial} if there exists a basis on $\mathbb{R}^d$ with respect to which $\phi(A)$ is a polynomial function of the entries of the matrix representation of $A$ and of the additional variable $(\det A)^{-1}$. We recall that a subset $Z$ of $GL_d(\mathbb{R})$ is called \emph{Zariski dense} if every polynomial $GL_d(\mathbb{R}) \to \mathbb{R}$ which vanishes on $Z$ also vanishes on $GL_d(\mathbb{R})$. We note that if a subset of $GL_d(\mathbb{R})$ is contained in a bounded subgroup of $GL_d(\mathbb{R})$ then it is not Zariski dense, since every bounded subgroup of $GL_d(\mathbb{R})$ is conjugate to $O(d)$ and therefore preserves an inner product, implying the existence of a polynomial which vanishes on the subgroup but not on $GL_d(\mathbb{R})$. We ask the following weaker form of Question \ref{qu:uq}:
\begin{question}
Let $\Gamma \subset GL_d(\mathbb{R})$ be a semigroup and let $s \in (0,d)$, and suppose that the equation \eqref{eq:fw} is satisfied for all $A \in \Gamma$. Is it possible that $\Gamma$ is Zariski dense?
\end{question}

\section{Acknowledgements}

This research was supported by the Leverhulme Trust (Research Project Grant number RPG-2016-194). The author thanks J. Bochi for helpful conversations.

\bibliographystyle{acm}
\bibliography{fox-like}

\begin{thebibliography}{10}

\bibitem{Ba07}
{\sc Bara\'nski, K.}
\newblock Hausdorff dimension of the limit sets of some planar geometric
  constructions.
\newblock {\em Adv. Math. 210}, 1 (2007), 215--245.

\bibitem{BaKaKo17}
{\sc B\'ar\'any, B., K\"aenm\"aki, A., and Koivusalo, H.}
\newblock Dimension of self-affine sets for fixed translation vectors.
\newblock arXiv:1611.09196, 2016.

\bibitem{BaRa17}
{\sc B\'ar\'any, B., and Rams, M.}
\newblock Dimension maximizing measures for self-affine systems.
\newblock {\em Trans. Amer. Math. Soc.\/} (2017).
\newblock To appear.

\bibitem{BoMo17}
{\sc Bochi, J., and Morris, I.~D.}
\newblock Equilibrium states of generalised singular value potentials and
  applications to affine iterated function systems.
\newblock arXiv:1710.04499, 2017.

\bibitem{CaFeHu08}
{\sc Cao, Y.-L., Feng, D.-J., and Huang, W.}
\newblock The thermodynamic formalism for sub-additive potentials.
\newblock {\em Discrete Contin. Dyn. Syst. 20}, 3 (2008), 639--657.

\bibitem{DaSi16}
{\sc Das, T., and Simmons, D.}
\newblock The {H}ausdorff and dynamical dimensions of self-affine sponges: a
  dimension gap result.
\newblock {\em Invent. Math. 210}, 1 (2017), 85--134.

\bibitem{FaKe16}
{\sc Falconer, K., and Kempton, T.}
\newblock Planar self-affine sets with equal {H}ausdorff, box and affinity
  dimensions.
\newblock {\em Ergodic Theory Dynam. Systems\/} (2017).
\newblock To appear.

\bibitem{FaMi07}
{\sc Falconer, K., and Miao, J.}
\newblock Dimensions of self-affine fractals and multifractals generated by
  upper-triangular matrices.
\newblock {\em Fractals 15}, 3 (2007), 289--299.

\bibitem{FaSl09}
{\sc Falconer, K., and Sloan, A.}
\newblock Continuity of subadditive pressure for self-affine sets.
\newblock {\em Real Anal. Exchange 34}, 2 (2009), 413--427.

\bibitem{Fa88}
{\sc Falconer, K.~J.}
\newblock The {H}ausdorff dimension of self-affine fractals.
\newblock {\em Math. Proc. Cambridge Philos. Soc. 103}, 2 (1988), 339--350.

\bibitem{Fe11}
{\sc Feng, D.-J.}
\newblock Equilibrium states for factor maps between subshifts.
\newblock {\em Adv. Math. 226}, 3 (2011), 2470--2502.

\bibitem{FeKa11}
{\sc Feng, D.-J., and K\"aenm\"aki, A.}
\newblock Equilibrium states of the pressure function for products of matrices.
\newblock {\em Discrete Contin. Dyn. Syst. 30}, 3 (2011), 699--708.

\bibitem{FeSh14}
{\sc Feng, D.-J., and Shmerkin, P.}
\newblock Non-conformal repellers and the continuity of pressure for matrix
  cocycles.
\newblock {\em Geom. Funct. Anal. 24}, 4 (2014), 1101--1128.

\bibitem{Fr12}
{\sc Fraser, J.~M.}
\newblock On the packing dimension of box-like self-affine sets in the plane.
\newblock {\em Nonlinearity 25}, 7 (2012), 2075--2092.

\bibitem{FrXX}
{\sc Fraser, J.~M.}
\newblock Remarks on the analyticity of subadditive pressure for products of
  triangular matrices.
\newblock {\em Monatsh. Math. 177}, 1 (2015), 53--65.

\bibitem{GoGu96}
{\sc Goldsheid, I.~Y., and Guivarc'h, Y.}
\newblock Zariski closure and the dimension of the {G}aussian law of the
  product of random matrices. {I}.
\newblock {\em Probab. Theory Related Fields 105}, 1 (1996), 109--142.

\bibitem{Hu81}
{\sc Hutchinson, J.~E.}
\newblock Fractals and self-similarity.
\newblock {\em Indiana Univ. Math. J. 30}, 5 (1981), 713--747.

\bibitem{Ka04b}
{\sc K\"aenm\"aki, A.}
\newblock On natural invariant measures on generalised iterated function
  systems.
\newblock {\em Ann. Acad. Sci. Fenn. Math. 29}, 2 (2004), 419--458.

\bibitem{KaMo16}
{\sc K\"aenm\"aki, A., and Morris, I.~D.}
\newblock Structure of equilibrium states on self-affine sets and strict
  monotonicity of affinity dimension.
\newblock {\em Proc. London Math. Soc.\/} (2017).
\newblock To appear.

\bibitem{KaVi10}
{\sc K\"aenm\"aki, A., and Vilppolainen, M.}
\newblock Dimension and measures on sub-self-affine sets.
\newblock {\em Monatsh. Math. 161}, 3 (2010), 271--293.

\bibitem{Mo16}
{\sc Morris, I.~D.}
\newblock An inequality for the matrix pressure function and applications.
\newblock {\em Adv. Math. 302\/} (2016), 280--308.

\bibitem{Mo17a}
{\sc Morris, I.~D.}
\newblock Ergodic properties of matrix equilibrium states.
\newblock {\em Ergodic Theory Dynam. Systems\/} (2017).
\newblock To appear.

\bibitem{Mo17b}
{\sc Morris, I.~D.}
\newblock A necessary and sufficient condition for a matrix equilibrium state
  to be mixing.
\newblock {\em Ergodic Theory Dynam. Systems\/} (2017).
\newblock To appear.

\bibitem{MoSh16}
{\sc Morris, I.~D., and Shmerkin, P.}
\newblock On equality of {H}ausdorff and affinity dimensions, via self-affine
  measures on positive subsystems.
\newblock {\em Trans. Amer. Math. Soc.\/} (2017).
\newblock To appear.

\bibitem{PrVo17}
{\sc Protasov, V.~Y., and Voynov, A.~S.}
\newblock Matrix semigroups with constant spectral radius.
\newblock {\em Linear Algebra Appl. 513\/} (2017), 376--408.

\end{thebibliography}

\end{document}